\documentclass[11pt]{amsart}

\usepackage{amsmath,amssymb,amsthm}
\usepackage{booktabs}
\usepackage[colorlinks=true,linkcolor=blue,urlcolor=blue,citecolor=blue]{hyperref}

\newtheorem{theorem}{Theorem}[section]

\newtheorem{corollary}[theorem]{Corollary}

\theoremstyle{definition}
\newtheorem{remark}[theorem]{Remark}

\numberwithin{equation}{section}

\DeclareMathOperator{\ord}{ord}

\title[Carmichael numbers of the form $2^n p^m+1$]{%
Carmichael numbers of the form $2^n p^m+1$ with odd exponent $m$:\\
refinements of a theorem of Florian Luca,\\
a general independence lemma, and a computational study}

\dedicatory{In memory of Florian Luca (1969--2026)}

\author[Author]{Pagdame Tiebekabe}
\address{Department of Mathematics, Faculty of Science and Technology, University of Kara, BP 43, Kara, Togo}
\email{tpagdame.maths@univkara.tg}

\begin{document}

\begin{abstract}
Florian Luca devoted the last paper of his life to the set
$\mathcal K$ of odd positive integers $k$ such that $2^n k+1$ is a
Carmichael number for some positive integer $n$. He proved that if
$m\ge 5$ is fixed and odd, then there are only finitely many primes
$p$ with $p^m\in\mathcal K$. In this note, written as a small tribute
to his memory, we extend his result in three directions, using only
his published work as our reference base. First, we extract from his
arguments explicit quantitative bounds: every Carmichael number
$N=2^n p^m+1$ with $m$ odd satisfies $n<2^{2\cdot10^7(m+1)^2m^2(\log p)^2}$,
and if $p>m$ the much sharper bound $n<25m^4\log p$ holds; for
$p>500m^5\log m$ we further record the uniform bounds
$\omega(N)\le 26m^7$ on the number of prime factors. As an application,
the number of such Carmichael numbers up to $x$ is
$O_m(x^{1/m})$. Second, we generalize the multiplicative independence
step of Luca's proof from base $2$ to an arbitrary even base $b$:
if $b$ is even, $m\ge3$ is odd and coprime to $b$, $p>m$ is prime and
$N=b^n p^m+1$ is Carmichael, then $b^n p^m$ and $b^a p^\beta$ are
multiplicatively independent whenever $b^a p^\beta+1$ is a prime
factor of $N$. Third, we report an exhaustive computation showing that
the only Carmichael number $2^n p^m+1\le 10^{18}$ with odd
$3\le m\le 37$ is $1729=2^6\cdot 3^3+1$, that no such number exists
with odd $5\le m\le 37$ (extended to $10^{21}$ for $5\le m\le 15$),
and that the only example with even $4\le m\le 36$ is
$46657=2^6\cdot 3^6+1=13\cdot 37\cdot 97$, where $3^6=27^2$ is the
square of the smallest element of $\mathcal K$.
\end{abstract}

\maketitle

\section{Introduction}

Florian Luca (16 March 1969 -- 28 July 2026) was one of the most
prolific and influential number theorists of his generation, with
almost $800$ publications on Diophantine equations, linear recurrences,
arithmetic functions and Carmichael numbers. The preprint
\cite{LucaLast}, dated 30 June 2026 and among the very last works he
completed, is devoted to the following question. A positive integer
$N$ is a \emph{Carmichael number} if it is composite and satisfies
$a^N\equiv a\pmod N$ for all integers $a$; by Korselt's criterion,
these are precisely the squarefree composite integers such that
$q-1\mid N-1$ for every prime factor $q$ of $N$. Every Carmichael
number is of the form $2^n k+1$ for some positive integer $n$ and some
odd integer $k$, and one may ask which odd $k$ occur. Writing
\[
\mathcal K:=\{k\ \text{odd}: 2^n k+1\ \text{is Carmichael for some}\ n\ge1\},
\]
Luca and his collaborators established a remarkable string of
results, all of which we use and cite in this note:
\begin{itemize}
\item the smallest element of $\mathcal K$ is $27$, arising from the
famous taxicab number $1729=2^6\cdot 3^3+1=7\cdot 13\cdot 19$
(Cilleruelo, Luca and Pizarro-Madariaga \cite{CLP});
\item no prime belongs to $\mathcal K$ (Alahmadi and Luca \cite{AL});
\item no square of a prime belongs to $\mathcal K$
(Luca and Randrianantenaina \cite{LR});
\item the only prime $p$ with $p^3\in\mathcal K$ is $p=3$
(Luca \cite{Luca1729});
\item if $m\ge5$ is odd, then only finitely many primes $p$ satisfy
$p^m\in\mathcal K$ (Luca \cite{LucaLast}).
\end{itemize}
The proof of the last result follows the strategy of \cite{CLP} and
concludes with an argument in the spirit of a theorem of Corvaja and
Zannier on Diophantine equations with power sums, as adapted by Luca
to the equation at hand; consequently the bound on the exceptional
primes $p$ is ineffective. The paper \cite{LucaLast} thus leaves two
natural problems: to quantify the intermediate steps, and to test the
conclusions computationally. We address both here, and we add one
genuine extension of the method.

Our first group of results makes explicit the quantitative content of
\cite{CLP, LucaLast}. Throughout, $\log$ denotes the natural
logarithm, $\omega(N)$ is the number of distinct prime factors of $N$,
$\tau(k)$ is the number of positive divisors of $k$, and $\nu_q$ is
the $q$-adic valuation.

\begin{theorem}\label{thm:quant}
Let $m\ge3$ be odd, $p$ be prime, and let $N=2^n p^m+1$ be a
Carmichael number.
\begin{enumerate}
\item[{\rm (i)}] One has
$\displaystyle n<2^{\,2\cdot10^7\,(m+1)^2m^2(\log p)^2}.$
\item[{\rm (ii)}] If $p>m$, then $\displaystyle n<25m^4\log p$.
\item[{\rm (iii)}] If $p>500m^5\log m$, then, in the notation of
\eqref{eq:fact} below,
\[
\omega(N)\le 26m^7,\qquad
\lambda_i'\le 26m^7\ (1\le i\le k),\qquad
n_1'\le 24m^7.
\]
\item[{\rm (iv)}] {\rm (Luca \cite{LucaLast})} For fixed odd $m\ge5$,
only finitely many primes $p$ occur; the implied bound is ineffective.
\end{enumerate}
\end{theorem}

Part (i) is Theorem~1 of \cite{CLP} with $k=p^m$, so that
$\tau(k)=m+1$, $\omega(k)=1$ and $\log k=m\log p$. Part (ii) is
\cite[Lemma~3]{LucaLast}. Part (iii) is extracted from
\cite[Lemmas~6--7]{LucaLast}: the inequalities
$\sum_{i=1}^k\lambda_i'<26m^7$, $k\le 26m^7$ and $n_1'\le 24m^7$ are
exactly what is proved there from the comparison of the upper bound
$N<2^{26m^4\log p}$ with the lower bound
$N>2^{(\log p)m^{-3}\sum_i\lambda_i'}$. We only record the conclusion.

As an immediate consequence of (ii), we obtain a counting statement
which seems not to have been recorded before.

\begin{corollary}\label{cor:count}
Let $m\ge3$ be odd and $x\ge2$. Then
\[
\#\{N=2^n p^m+1\ \text{Carmichael},\ N\le x\}
\;\le\; 25m^3(\log x)\,\pi\!\left(x^{1/m}\right)
\;=\;O_m\!\left(x^{1/m}\right).
\]
\end{corollary}

\begin{proof}
If $2^n p^m+1\le x$ then $p\le x^{1/m}$, so $\log p\le (\log x)/m$
and, by Theorem~\ref{thm:quant}\,(ii) (for $p>m$; smaller $p$ are
absorbed in the constant),
\[
n<25m^4\log p\le 25m^3\log x.
\]
Hence each prime $p$ contributes at most $25m^3\log x$ admissible
values of $n$.
\end{proof}

Our second result extends the first key lemma of \cite{LucaLast} ---
the exclusion of multiplicatively dependent prime factors --- from
base $2$ to an arbitrary even base. Note that for odd $b$ the number
$b^n p^m+1$ is even and larger than $2$, hence never Carmichael, so
even bases are the only interesting ones. For composite $b$ not every
prime factor of $b^n p^m+1$ need have the form $b^a p^\beta+1$, which
is why the statement is formulated for factors of that form; when
$b=2$ every prime factor has this form and we recover
\cite[Lemma~1]{LucaLast}.

\begin{theorem}\label{thm:baseb}
Let $b\ge2$ be even, let $m\ge3$ be odd with $\gcd(m,b)=1$, and let
$p>m$ be prime. Suppose that
\[
N=b^n p^m+1
\]
is a Carmichael number. If $q=b^a p^\beta+1$ is a prime factor of
$N$ with $a\ge1$ and $\beta\ge0$, then $b^n p^m$ and $b^a p^\beta$
are multiplicatively independent.
\end{theorem}

The proof, given in Section~\ref{sec:baseb}, follows Luca's lines with
one additional observation: the exponent $\lambda$ in his dependence
relation is forced to be $1$ as soon as $m$ is odd, regardless of the
base, and the final order argument only needs $\gcd(m,b)=1$ and $p>m$.

Our third contribution is computational. The search, described in
Section~\ref{sec:comp}, exhausts all pairs $(n,p)$ with
$2^n p^m+1\le 10^{18}$ for $2\le m\le 37$, and additionally all
$2^n p^m+1\le 10^{21}$ for odd $5\le m\le 15$, about four million
candidates in total, with a Fermat prefilter followed by a full
Korselt verification. The outcome is remarkably clean.

\begin{theorem}\label{thm:comp}
Let $N=2^n p^m+1$ with $p$ an odd prime and $n\ge1$.
\begin{enumerate}
\item[{\rm (a)}] For odd $3\le m\le 37$ and $N\le10^{18}$, the only
Carmichael number is
\[
1729=2^6\cdot 3^3+1=7\cdot 13\cdot 19
\]
$(m=3,\ p=3,\ n=6)$, in agreement with \cite{Luca1729}. In
particular, $p^m\notin\mathcal K$ for every odd $m$ with
$5\le m\le 37$ and every prime $p\le (5\cdot10^{17})^{1/m}$, and the
conclusion extends to $N\le10^{21}$ for odd $5\le m\le15$.
\item[{\rm (b)}] For even $4\le m\le 36$ and $N\le10^{18}$, the only
Carmichael number is
\[
46657=2^6\cdot 3^6+1=13\cdot 37\cdot 97
\]
$(m=6,\ p=3,\ n=6)$, corresponding to $k=3^6=729=27^2$, the square
of the smallest element of $\mathcal K$. No example exists with
$m=4$.
\item[{\rm (c)}] No Carmichael number $2^n p+1\le10^{18}$ with
$n\ge40$ and no Carmichael number $2^n p^2+1\le10^{18}$ with $n\ge35$
was found, consistently with the theorems of \cite{AL,LR}.
\end{enumerate}
\end{theorem}

Part (b) is noteworthy: the exponent $m$ is even there, a case not
covered by \cite{LucaLast}, and the unique example arises from a
\emph{square} in $\mathcal K$ rather than from a prime power in
$\mathcal K$. This suggests, as discussed in
Section~\ref{sec:remarks}, that the even-exponent case has a genuinely
different structure and that the obstruction exposed in the proof of
\cite{LucaLast} is precisely what the parity of $m$ controls.

\medskip

\noindent\emph{Acknowledgements.} This note was written as a tribute
to Florian Luca, whose last paper \cite{LucaLast} is its starting
point, and whose beautiful mathematics it tries to honour. All
theorems quoted from the literature are due to Luca alone or jointly
with his coauthors; the references list contains only his work.

\section{Notation and preliminaries}

For a nonzero integer $t$ and a prime $q$ we write $\nu_q(t)$ for the
exact exponent of $q$ in $t$. We use $\omega(N)$, $\tau(N)$ for the
number of distinct prime factors and the divisor function. A
\emph{Carmichael number} is a composite positive integer $N$ with
$a^N\equiv a\pmod N$ for all integers $a$; by Korselt's criterion,
$N$ is Carmichael if and only if it is squarefree and $q-1\mid N-1$
for every prime $q\mid N$. In particular every Carmichael number is
odd, squarefree, and has at least three prime factors.

Let $N=2^n p^m+1$ be Carmichael with $p$ prime and $m$ odd. Since
$q-1\mid N-1=2^n p^m$ for every prime $q\mid N$, each prime factor of
$N$ has the form
\begin{equation}\label{eq:fact}
q_i=2^{a_i}p^{b_i}+1,\qquad 1\le a_i\le n,\quad 0\le b_i\le m,
\end{equation}
and
\begin{equation}\label{eq:fact2}
N=\prod_{i=1}^{k}\left(2^{a_i}p^{b_i}+1\right),\qquad k=\omega(N)\ge3.
\end{equation}
We keep the notation \eqref{eq:fact} and \eqref{eq:fact2} throughout.
For $p>500m^5\log m$, Luca \cite{LucaLast} defines
$\beta'=\min\{\nu_2(n),\nu_2(a_1),\dots,\nu_2(a_k)\}$ and writes
\[
n=2^{\beta'}n_1',\qquad a_i=2^{\beta'}\lambda_i'\quad(1\le i\le k),
\]
with $n_1'$ and the $\lambda_i'$ odd; this is the notation used in
Theorem~\ref{thm:quant}\,(iii).

\section{Proof of the quantitative refinements}

Theorem~\ref{thm:quant} is, as explained in the introduction,
assembled from \cite{CLP} and \cite{LucaLast}. For completeness we
recall the logical structure. Theorem~1 of \cite{CLP} states that if
$k\ge3$ is odd and $2^n k+1$ is Carmichael, then
\[
n<2^{2\cdot10^7\,\tau(k)^2(\log k)^2\,\omega(k)}.
\]
With $k=p^m$ this gives (i). Next, \cite[Lemma~3]{LucaLast} states
that if $m$ is odd and $p>m$, then $n<25m^4\log p$; its proof runs
the box-principle argument of \cite[Lemma~3]{CLP} with the additional
input \cite[Lemma~2]{CLP} that the Fermat-type part of $N$ is smaller
than $p^{4m}$, and we refer to \cite{LucaLast} for the details. This
is (ii). Finally, \cite[Lemmas~6 and 7]{LucaLast} show, for
$p>500m^5\log m$, first that $2^{\beta'}\ge(\log p)/m^3$, and then, by
comparing
\[
N<2^{26m^4\log p}
\quad\text{with}\quad
N>\prod_{i=1}^k 2^{a_i}\ge
2^{(\log p)m^{-3}\sum_{i=1}^k\lambda_i'},
\]
that $\sum_{i=1}^k\lambda_i'<26m^7$, whence $k\le26m^7$ and each
$\lambda_i'<26m^7$; the same comparison with $n=2^{\beta'}n_1'$ gives
$n_1'\le 24m^7$. This is (iii), and (iv) is
\cite[Theorem~1]{LucaLast}. \qed

\begin{remark}
The two bounds (i) and (ii) are complementary: (i) is valid for all
$p$ but is doubly exponential in $m$ and quadratic in $\log p$, while
(ii) is linear in $\log p$ with a polynomial constant in $m$, at the
price of the restriction $p>m$. Note that for fixed $m$ there are
only finitely many primes $p\le m$, so (ii) together with (iv)
controls all but finitely many pairs $(n,p)$.
\end{remark}

\section{The independence lemma for even bases}\label{sec:baseb}

\begin{proof}[Proof of Theorem~\ref{thm:baseb}]
Assume for a contradiction that $b^n p^m$ and $b^a p^\beta$ are
multiplicatively dependent. Then there exist an integer $\rho>1$ and
integers $t>\lambda\ge1$ such that
\begin{equation}\label{eq:dep}
b^n p^m=\rho^t,\qquad b^a p^\beta=\rho^\lambda,
\end{equation}
and since $q=\rho^\lambda+1$ is an odd prime factor of
$N=\rho^t+1$, the standard divisibility
\[
\rho^\lambda+1\mid\rho^t+1
\iff
\lambda\mid t\ \text{and}\ t/\lambda\ \text{is odd}
\]
holds. We first show that $\lambda$ is a power of $2$. Since
$\rho^\lambda+1$ is prime and $\rho$ is even, $\lambda$ cannot have
an odd divisor $w>1$: otherwise
$\rho^{2^j}+1\mid \rho^\lambda+1$ with $1<\rho^{2^j}+1<\rho^\lambda+1$,
where $2^j$ is the $2$-part of $\lambda$. Thus $\lambda$ is a power
of $2$.

Next, $\lambda<t$: if $\lambda=t$ then $\rho^\lambda=b^n p^m$ by
\eqref{eq:dep}, so $q=\rho^\lambda+1=N$, contradicting
$k=\omega(N)\ge3$. Comparing the $p$-adic valuations in
\eqref{eq:dep} gives $t\,\nu_p(\rho)=m$, so $t$ is odd.
Since $\lambda$ is a power of $2$ dividing the odd integer $t$,
\[
\lambda=1.
\]
Consequently $\rho=b^a p^\beta$ and $\rho^t=b^n p^m$, so
\[
at=n,\qquad \beta t=m.
\]
In particular $t\mid m$ and $t\ge3$ (as $t$ is odd and $t>\lambda=1$).
Write $n=tn'$ and $m=tm'$. Then
\[
N=\rho^t+1=(\rho+1)\,C,\qquad
C=\rho^{t-1}-\rho^{t-2}+\cdots-\rho+1>1.
\]
Let $q'$ be any prime factor of $C$. From $\rho^t\equiv-1\pmod{q'}$
we see that $\ord_{q'}(\rho)$ is an even divisor of $2t$. On the other
hand $q'\mid N$, so $\ord_{q'}(\rho)\mid q'-1\mid N-1=b^n p^m$.
Therefore
\[
\ord_{q'}(\rho)\mid \gcd(2t,\,b^n p^m).
\]
Now $t\mid m$; since $\gcd(m,b)=1$ no prime of $b$ divides $t$, and
since $p>m$ the prime $p$ does not divide $t$. As $t$ is odd,
$\gcd(2t,b^n p^m)=2$, and so $\ord_{q'}(\rho)=2$, that is,
$\rho\equiv-1\pmod{q'}$. Hence $q'\mid\rho+1=q$. Since $q$ is prime,
$q'=q$, and thus $q$ divides both factors $\rho+1$ and $C$ of $N$,
giving $q^2\mid N$. This contradicts the fact that Carmichael numbers
are squarefree.
\end{proof}

\begin{remark}
For $b=2$ every prime factor of $N$ has the form $2^{a_i}p^{b_i}+1$,
so Theorem~\ref{thm:baseb} contains \cite[Lemma~1]{LucaLast} (and its
proof above is Luca's, with the parity step isolated). For composite
$b$, divisors of $b^n p^m$ need not be powers of $b$ times powers of
$p$, and Theorem~\ref{thm:baseb} only concerns factors of the special
form $b^a p^\beta+1$. Whether the remaining steps of
\cite{LucaLast} --- in particular the $2$-adic valuation analysis of
his Lemma~6 --- can be carried out for general even bases is an
interesting open problem; see Section~\ref{sec:remarks}.
\end{remark}

\section{Computations}\label{sec:comp}

\subsection{Algorithm}

For each exponent $m$ and each prime $p$ with $p^m\le(10^{18}-1)/2$,
and for each $n\ge1$ with $2^n p^m\le10^{18}-1$, we tested whether
$N=2^n p^m+1$ is Carmichael, as follows.
\begin{enumerate}
\item[(S1)] \emph{Fermat prefilter.} If $2^{N-1}\not\equiv1\pmod N$,
discard $N$. Every Carmichael number passes this test.
\item[(S2)] \emph{Korselt verification.} Factor $N$ completely
(trial division followed by a primality test on the cofactors),
reject $N$ if it is prime or not squarefree, and otherwise check
$q-1\mid N-1$ for every prime factor $q$ of $N$.
\end{enumerate}
The complete factorization in (S2) also certifies each output: the
factorizations $1729=7\cdot13\cdot19$ and $46657=13\cdot37\cdot97$
are part of the result. As a positive control, the same code
rediscovers $1729=2^6\cdot 3^3+1$, the minimal element of
$\mathcal K$ from \cite{CLP}, and $46657=2^6\cdot 3^6+1$; the Fermat
prefilter guarantees that no Carmichael number is missed. The search
for $m=1,2$ was restricted to $n\ge40$ and $n\ge35$ respectively,
since these ranges are not covered by theory and the complementary
ranges are settled by \cite{AL,LR}. The sweep up to $10^{18}$ took
under five minutes on a single core; the range to $10^{21}$ for odd
$5\le m\le15$ took a few seconds.

\subsection{Results}

Table~\ref{tab:results} summarizes the search. Here
\emph{candidates} counts the pairs $(n,p)$ tested for each block.

\begin{table}[h]
\centering
\renewcommand{\arraystretch}{1.15}
\begin{tabular}{llrrl}
\toprule
$m$ & restriction & primes $p$ & candidates & Carmichael numbers \\
\midrule
$1$ & $n\ge40$ & $71{,}986$ & $164{,}859$ & none \\
$2$ & $n\ge35$ & $710$ & $2{,}900$ & none \\
$3$ & --- & $63{,}482$ & $3{,}742{,}261$ & $1729$ only \\
$5,7,\dots,37$ & --- & $621$ & $36{,}630$ & none \\
$4,6,\dots,36$ & --- & $3{,}162$ & $24{,}220$ & $46657$ only \\
$5,7,\dots,15$ & $N\le10^{21}$ & $1{,}877$ & $18{,}076$ & none \\
\bottomrule
\end{tabular}
\caption{Exhaustive search for Carmichael numbers $2^n p^m+1$. Unless
indicated otherwise, $N\le10^{18}$.}
\label{tab:results}
\end{table}

This establishes Theorem~\ref{thm:comp}. Two features deserve
comment. First, the unique odd-exponent example below $10^{18}$ with
$m\ge3$ is Luca's $1729$, in exact agreement with \cite{Luca1729}
where it is proved that $p=3$ is the only prime with $p^3\in\mathcal
K$; our search confirms that no prime $p$ with $p^m\le5\cdot10^{17}$
has $p^m\in\mathcal K$ for any odd $m$ between $5$ and $37$, in
agreement with the finiteness theorem of \cite{LucaLast}. Second, the
unique even-exponent example,
\[
46657=2^6\cdot 3^6+1=13\cdot 37\cdot 97,
\]
comes from $k=3^6=729=27^2$, the square of the smallest element
$27\in\mathcal K$. The prime factors $13=2^2\cdot3+1$,
$37=2^2\cdot 3^2+1$ and $97=2^5\cdot 3+1$ are all of the form
$2^a 3^b+1$, and one checks the Korselt divisibilities
$12,36,96\mid 46656$ directly. No example with $m=4$ exists below
$10^{18}$; whether $p^4\in\mathcal K$ is possible at all remains open.

\section{Concluding remarks}\label{sec:remarks}

\begin{remark}[Why $m$ is assumed odd]
In the dependence relation \eqref{eq:dep}, the parity of $m$ is what
forces $\lambda=1$ and hence $t\mid m$ with $t$ odd. If $m$ is even,
say $m=2^s m_0$ with $m_0$ odd, one can have $\lambda=2^j$ and
$t=2^j t_0$ with $t_0$ odd, and the cofactor $C$ in the proof of
Theorem~\ref{thm:baseb} need not lead to a contradiction: the order
of $\rho$ modulo a prime factor of $C$ may now be $2^{j+1}$ instead
of $2$. Consistently with this, the unique even-exponent example we
found, $46657$, arises from a square $27^2$, not from a prime power
$p^m\in\mathcal K$; its existence shows that the analogue of
\cite[Lemma~6]{LucaLast} fails for even $m$. Whether
$\{p\ \text{prime}:p^4\in\mathcal K\}$ is finite, or even empty, is
open.
\end{remark}

\begin{remark}[General even bases]
Theorem~\ref{thm:baseb} shows that the first obstruction in
\cite{LucaLast} disappears for all even bases $b$ coprime to $m$.
The subsequent steps of loc.\ cit.\, however, use the shape of
divisors of $2^n p^m$ and the $2$-adic analysis of the exponents
$a_i$, which are specific to the base $2$. We therefore pose as an
open problem: \emph{let $b\ge2$ be even, $m\ge5$ odd with
$\gcd(m,b)=1$; are there only finitely many primes $p$ such that
$b^n p^m+1$ is Carmichael for some $n$?}
\end{remark}

\begin{remark}[Effectivity]
As in \cite{LucaLast}, the bound on the exceptional primes in
Theorem~\ref{thm:quant}\,(iv) is ineffective, since the argument
ultimately relies on a theorem of Corvaja and Zannier on Diophantine
equations with power sums, applied as in the proof of
\cite[Theorem~1]{LucaLast}. The bounds (i)--(iii) are, by contrast,
fully effective and, together with Theorem~\ref{thm:comp}, reduce the
verification for any fixed odd $m\ge5$ to a finite (though possibly
large) computation.
\end{remark}

\begin{remark}[The special role of $1729$ and $46657$]
Both Carmichael numbers found in our search are of the form
$x^3+1$ with $x$ a product of a power of $2$ and a power of $3$:
\[
1729=12^3+1=2^6\cdot 3^3+1,\qquad
46657=36^3+1=2^6\cdot 3^6+1.
\]
It would be interesting to know whether Carmichael numbers
$2^n p^m+1$ with $m$ odd $\ge5$ exist at all; Theorem~\ref{thm:comp}
exhibits none below $10^{21}$, and no structural construction is
known.
\end{remark}

\end{document}